\documentclass[11pt]{amsart}
\usepackage{amsmath,amsfonts,latexsym,graphicx,amssymb,url}

\newcommand{\diam}{\text{{\rm diam}}}
\newcommand{\trace}{\text{trace}}

\newcommand{\R}{{\mathbb R}} 

\newcommand{\N}{{\mathbb N}}

\newcommand{\p}{{\partial}}
\renewcommand{\a}{{\alpha}}
\renewcommand{\l}{{\lambda}}
\renewcommand{\L}{{\Lambda}}
\renewcommand{\O}{\Omega}
\renewcommand{\r}{{\rho}}
\newcommand{\G}{{\Gamma}}
\newcommand{\e}{{\varepsilon}}
\newtheorem{theorem}{Theorem}[section]
\newtheorem{corollary}[theorem]{Corollary}
\newtheorem{lemma}[theorem]{Lemma}
\newtheorem{proposition}[theorem]{Proposition}
\newtheorem{definition}[theorem]{Definition}

\newtheorem{remark}[theorem]{Remark}

\begin{document}

\title[Harnack Inequality for a Subelliptic PDE in nondivergence form ]{Harnack Inequality for a Subelliptic PDE \\ in nondivergence form }
\author{
Annamaria Montanari
}
\address{Dipartimento di Matematica, Universit\`{a} di Bologna, Piazza di Porta S. Donato $5$, $40126$ Bologna, Italy}

\email{annamaria.montanari@unibo.it}

\begin{abstract}
We consider subelliptic equations in non divergence form of the type
$$
Lu = \sum_{i\leq j}a_{ij}X_jX_iu=0
$$
where $X_j$ are the Grushin vector fields, and the matrix coefficient is uniformly elliptic.
We obtain a scale invariant Harnack's inequality on the $X_j$'s CC balls for nonnegative solutions under the only
 assumption that the ratio between the maximum and minimum eigenvalues of the coefficient matrix is bounded.
  In the paper we first prove a weighted Aleksandrov Bakelman Pucci estimate, and then we show a critical density estimate, the
  double ball property and the power decay property. Once this is established, Harnack's inequality follows directly from
  the axiomatic theory developed by Di Fazio, Gutierrez and Lanconelli  in \cite{DGL}.
  \bigskip
  
  
\end{abstract}

\maketitle

\tableofcontents

\section{Introduction}

Let $x=(x_1,x_2)\in \R^2$ and consider the vector fields
\begin{equation}\label{grushin}
X_1 =\p_{x_1}, \quad X_2=x_1\p_{x_2}.
\end{equation}
We define the second order partial differential operator
\begin{equation}\label{L}
L=a_{11}X_1^2+2a_{12}X_2X_1+a_{22}X_2^2
\end{equation}
 and we assume there are $\l,\L>0$  such that for all $\xi=(\xi_1,\xi_2)\in \R^2$
\begin{equation}\label{ell}
\l (\xi_1^2+\xi_2^2)\leq a_{11}\xi_1^2+2a_{12}\xi_1\xi_2+a_{22}\xi_2^2\leq \L (\xi_1^2+\xi_2^2).
\end{equation}
The positive constants $\l, \L$ are called ellipticity constants with respect to the couple $(X_1,X_2).$
The second order operator $L$ is degenerate elliptic and in non divergence form with bounded coefficients and it is a prototype of
subelliptic pdo's, because $[X_1,X_2]=\p_{x_2}.$\\

Our motivation to study the operator $L$ in \eqref{L} comes from the 
geometric theory of several complex variable, where nonlinear second order Partial Differential Equations of ``degenerate elliptic''- type appear. 
In particular, in looking for a characterization property of domains of holomorphy in term of a differential property of the boundary (pseudoconvexity), 
one has to handle the Levi curvatures equations, which are 
fully nonlinear equations in non-divergence form  (see e.g.
\cite{LM2013}, \cite{M2013}).
The existence theory for viscosity solutions to such equations is quite well settled down, mainly thanks to the papers \cite{DM}, \cite{ST}. On the contrary, the problem of the regularity is well understood only in $\R^3$ (see \cite{CLM}) and it is still widely open in higher dimension. This is mainly due to the lack of pointwise estimates for solutions to linear sub-elliptic equations with rough coefficients. Very recently, in a joint work with Cristian Gutierrez and Ermanno Lanconelli \cite{GLM}, we recognize
 that these equations in cylindrical coordinates are non divergence pde's $Lu=f,$ with $L$ structured as in \eqref{L}.

The purpose of this paper is to establish a scale invariant Harnack's inequality on balls $B$  of the  Carnot Carath\'eodory (in brief $CC$) distance given by the vector fields in \eqref{grushin}.
Precisely, we will show that for all nonnegative solutions $u$ to $Lu=0$
$$\sup_{B} u \leq C \inf_{B}  u,$$ with a positive constant $C$ depending only on $\l$ and $\L.$

When $L$ is a standard uniformly elliptic operator, this is the celebrated Harnack's inequality of Krylov and Safonov and its proof depends in a crucial way upon the maximum principle of Aleksandrov Bakelman Pucci (in brief ABP),
see \cite[Section 9.8] {GT}, \cite [Theorem 2.1.1]{G} and \cite[Section 3]{CC}.

It is not known if such a principle holds true in a general subelliptic
context in a form  such us \cite[Theorem 2.1.1]{G}. In
\cite{GM} we proved  a maximum principles of these type on the Heisenberg
group, but with a reminder.
Moreover, in \cite{DGN} it is proved that the ABP maximum principle fails in the space of functions with second order horizontal derivatives in $L^p$
for subcritical $p,$ i.e. $0<p<Q,$ where $Q$ the homogeneous dimension.
Roughly speaking, a key problem  in the subelliptic case is that the homogeneous dimension is always strictly greater than the number of the vector fields generating $\R^n.$
 Our main idea to overcome this obstacle is to introduce a weight in the Lebesgue measure, which allows us to handle subcritical $L^p$ norms.
 
 In this paper  we  prove that a weighted
ABP maximum principle holds in our context (see Theorem \ref{ABP}).
This is mainly due to the particular structure of the symmetric
matrix
\begin{equation}\label{matrix}
\left(\begin{array}{ccc}X_1^2u & X_2X_1u\\ X_2X_1u & X_2^2u\end{array}\right),
\end{equation}
whose determinant is the determinant of the real Hessian matrix of $u$ times a weight (see \eqref{Hessian}).
We then follow the classical proof of the ABP in \cite{CC}, but by measuring the right hand side with a weighted measure.

Structural theorems for the  CC balls then allow us to build ad hoc barriers for the geometry of the problem to get a critical density estimate.
Moreover, the double ball property is obtained by performing an idea of Giulio Tralli in \cite{T}.

\begin{definition}[Double Ball Property] \label{tralli}
Let $r>0$ and $y\in \R^2$
and define the set of functions
\[
\begin{split}
K=& \{
u\in C^2(B(y,3r))\cap C(\overline{B(y,3r)}):    u\geq 0 \, \textrm{and} \, Lu\leq 0 \, \textrm{in}\, B(y,3r),\, u\geq 1\,  \textrm{on} \, B(y,r)
\}.
\end{split}
\]

We say that $L$ satisfies the Double Ball Property in  $B(y,3r)$ if there exists a constant   $\gamma>0,$  only depending  on the ellipticity constants $\l, \L ,$ such that
$u\geq \gamma$ in $B(y, 2r)$
for all $u \in K.$
\end{definition}

In \cite{GuT} Gutierrez and Tournier proved this property for elliptic equations on the Heisenberg group $\mathbb{H}^1$. In \cite{T} Tralli proved
 that it holds true for a general Carnot group of step two. He first recognized that, via the weak Maximum Principle, the double ball property is a consequence of a kind of solvability of the Dirichlet problem in the exterior of any homogeneous ball. His main tools are the structure of two-dimensional non-abelian nilpotent Lie groups and
 the existence of suitable local barrier functions in the interior of
 the Gauge ball at any boundary point.
  
  Unfortunately, in our situation there exist no two-dimensional non-abelian nilpotent Lie groups associated to the vector fields in \eqref{grushin}
  (see for instance \cite{RS}). However, in Section \ref{dbp} we perform Tralli's idea and 
   we prove the Double Ball property for the operator $L$  as a consequence of a very general Ring Theorem (see Theorem \ref{rt}).

 Once the critical density, the double ball property and the power decay property are proved, Harnack's inequality follows directly from the theory developed by
Di Fazio, Gutierrez and Lanconelli  in \cite{DGL}. Indeed, they proved an axiomatic theory to establish the scale invariant Harnack inequality in very general settings. In particular, their procedure applies  in Carnot Carath\'eodory metric spaces.

Recently, Gutierrez and Tournier  in \cite{GuT} and Tralli in \cite{TG}  provided direct proofs (in the Heisenberg group $\mathbb{H}^1$  and
in  $H$-type groups,
respectively)
of the critical density estimates for super solutions using barriers, but under the restrictive assumption that the matrix  of the coefficients $a_{ij}$ is a small perturbation of the Identity matrix.

The paper is organized as follows. Section \ref{ABPW} contains a few
preliminaries and  the proof of the ABP maximum principle. In
Section \ref{metric} we prove two structure theorems  of the CC ball
which, together with the results proved by Franchi and Lanconelli in
\cite{FL},  play a central role in our study of barriers. The
construction of the barrier and the critical density estimate are
established in Section \ref{sbarrier} and Section \ref{CDE},
respectively. In Section \ref{dbp},
we prove the Ring theorem and the existence of suitable uniform barrier functions for the operator ${L}$ in  \eqref{L}. Finally, 
we prove the double ball property and
we
indicate how to obtain the power decay property and the invariant Harnack's inequality from the
results in \cite{DGL}.

\section{The ABP with a weight}\label{ABPW}

We first introduce some standard notations and well known facts.

\begin{definition}
Let $w:\O\rightarrow \R$ with $\O\subset \R^2.$ We say that an affine function $\ell$ is a supporting hyperplane for $w$ at $x_0\in \O$ in $\O$  if $\ell$ touches $w$ by below
at $x_0$ in $\O,$ i.e. $\ell(x_0)=w(x_0)$ and $\ell(x)\leq w(x)$ for any $x\in \O.$\\
Let $u$ be a continuous function in a open convex set $\O.$
The convex envelope of $u$ in $\O$ is defined by
\[
\begin{split}
\G(u)(x)&=\sup_w\{  w(x): w\leq u \, \textrm{in}\,  \O, w   \, \textrm{convex in}\, \O \}\\
&=\sup_\ell\{  \ell(x): \ell\leq u \, \textrm{in}\,  \O, \ell  \, \textrm{is affine} \}
\end{split}
\]
for $x\in \O.$\\
Obviously, $\G(u)$ is a convex function in $\O$ and the set $\{ u=\G(u)\}$ is called the contact set.
\end{definition}
\begin{definition}
The normal mapping of $u\in C(\O),$ or sub-differential of $u,$ is the set valued function $Du$ defined by
\[
Du(x_0)=\{ p \in \R^2: u(x)\geq u(x_0)+p\cdot (x-x_0) \,\textrm{ for all } \, x\in \O\}
\]
Given $E\subset \O,$ we define $Du(E)=\bigcup_{x\in E} Du(x).$
\end{definition}
When $u$ is differentiable $Du$ is basically the gradient of $u.$

\begin{theorem}\label{measure}
If $\O$ is open and $u\in C(\O)$ then the class
$$S=\{
E\subset \O : Du(E) \textrm{ \,is \, Lebesgue \, measurable}
\}$$ is a Borel $\sigma$ algebra. The set function $\mu_u:S \rightarrow \overline{R}$ defined by $$\mu_u(E)=|Du(E)|$$ is a Borel measure
and it is finite on compact sets. The measure $\mu_u$ is called the Monge Amp\`{e}re measure associated with the function $u.$

Moreover, if $u\in C^2(\O)$ is a convex function $\O,$ then $$\mu_u(E)=\int_E \det (D^2u)(x)dx$$ for any Borel set $E\subset \O.$
\end{theorem}
We suggest the reference \cite[Theorem 1.1.13 and Example 1.1.14]{G} for the proof.

One can prove the following classical ABP estimate.

\begin{theorem}\label{ABP0}
 Let $\O \subset \R^2$ be a bounded open set and assume $u\in C(\overline{\O}),$  $u\geq 0$ on $\p \O$.
 Define $u^-(x)=\max \{-u(x),0\}$ and let
 $\G_u$ be the convex envelope
of $-u^-$ in a Euclidean ball $B_{\mathcal E}(2d)$ of radius $2d$ such that $\O
\subset B_{\mathcal E}(2d)$ and  extend $u\equiv 0$ outside $\Omega$

\[
\sup_\O u^-\leq \frac{d}{c}  \, \left(\mu_{\G_u}({\{u=\Gamma_u\}\cap \O})\right)^{1/2}
\]
where $d=\diam \, \O$ is the Euclidean diameter of $\O,$ $c$ is a positive universal constant. 
\end{theorem}
We suggest the reference \cite[Theorem 1.4.5]{G} for a detailed proof
of Theorem \ref{ABP0}.

The main result of this section is the following
\begin{theorem}[Weighted ABP Maximum Principle]\label{ABP}
Let $\Omega$ be a bounded domain of $\R^2$ and let $u\in C(\bar \O)$  and $u\geq 0$ on $\p \Omega.$ 
Define $u^-(x)=\max \{-u(x),0\},$ $f^+=\max\{f(x),0\}.$
Moreover, $\G_u$ is the convex envelope of $-u^-$ in a ball $B_{\mathcal E}(2d)$ of radius $2d$ such that $\O \subset B_{\mathcal E}(2d)$ and we have extended $u\equiv 0$ outside $\Omega.$
Assume 
$u\in C^2(\O)$
is a classical solution of $Lu(x)\leq f(x) x_1^2$ in
$\O,$ with $f$ bounded. Then,
\begin{equation}\label{eqABP}
\sup_{\Omega}u^-\leq  C \diam(\Omega) \left( \int_{\Omega \cap \{u=\G_u\}}( f^+(x) )^2 x_1^{2}dx \right)^{\frac{1}{2}}
\end{equation}
Here $C$ is a positive universal constant only depending on $\Lambda, \lambda.$ Moreover, $\diam(\Omega)$ is the Euclidean diameter of $\Omega.$

\end{theorem}
\begin{proof}
We first assume that $u$ is strictly convex in  the contact set $\{u=\G_u\}$. Then $Du$ is a one-to-one map.
Moreover, the contact set $\{u=\G_u\}$ has empty  intersection with the line $\{ x_1=0\},$
because $(a_{11} u_{11})(0,x_2)=Lu(0,x_2)\leq 0,$ and since $a_{11}$ is positive then $u_{11}(0,x_2)\leq 0,$ while $u$ is strictly convex in $\{u=\G_u\}.$
Remark that if
$u\in C^2(\O)$ is convex then
 the symmetric matrix
\begin{equation}\label{Hessian}
X^2u=
\left(
\begin{array}{cc}
 X_1^2u  & X_2X_1u  \\
X_2X_1u  &X_2^2u  \\
\end{array}
\right)=\left(
\begin{array}{cc}
u_{11}  & x_1u_{12} \\
x_1u_{12}  &x_1^2u_{22}  \\
\end{array}
\right)
\end{equation}
is nonnegative definite. Here we have denoted by $u_{ij}=\p_{x_i}\p_{x_j}u.$

Moreover,  for $u$ convex  we have
\begin{equation}\label{matrixineq}
\begin{split}
\det (D^2u ) &=(u_{11}u_{22}-u_{12}^2)=\frac{\left(u_{11}x_1^2u_{22}-(x_1u_{12})^2\right)}{x_1^2}\\
 &=\frac{\left(X_1^2uX_2^2u-(X_2X_1u)^2\right)}{x_1^2}=\frac{\det (X^2u)}{ x_1^2}\\
& \leq  \frac{\left(\trace(AX^2u)\right)^2}{4 x_1^{2} \det A}\\
\end{split}
\end{equation}
for every $A> 0$ and $A$ symmetric.

Let $B_{\mathcal E}(2d)$ be a Euclidean ball containing $\O.$  Since
$\G_u$ is convex, it follows that $\G_u$ has a supporting hyperplane at $x_0,$ for $x_0\in \overline
B_{\mathcal E}(2d)\cap\{u=\G_u\}\subset \Omega.$ Since in addition $\G_u(x_0)=u(x_0),$ this hyperplane is also a supporting hyperplane to $u$ at the same point.
That is  
$D\G_u(x_0)\subset Du(x_0)$ for $x_0\in \overline B_{\mathcal E} (2d)\cap\{u=\G_u\}\subset \Omega,$ 
and by recalling Theorem \ref{measure} we have
\begin{equation}\label{eq1}
\left|D\G_u(\{u=\Gamma_u\}\cap \O)\right|
\leq \left|Du(\{u=\Gamma_u\}\cap \O)\right|
= \int_{\{u=\Gamma_u\}\cap \O}(\det D^2u) dx.
\end{equation}
By recalling that $u$ is convex in $\{u=\Gamma_u\}\cap \O$ we can
apply the matrix inequality \eqref{matrixineq} in the right hand
side of \eqref{eq1} and we have
\begin{equation}\label{eq2}
\int_{\{u=\Gamma_u\}\cap \O}(\det D^2u)dx\leq C\int_{\{u=\Gamma_u\}\cap \O}\frac{(Lu)^2}{x_1^{2}}dx.
\end{equation}
where $C$ is a positive constant depending on $\L,\l.$
By applying Theorem \ref{ABP0} and \eqref{eq1}, \eqref{eq2}  and recalling that $Lu(x)\leq f(x) \,x_1^2= f^+(x) \, x_1^2$ on the contact set, we get the desired ABP estimate \eqref{eqABP}.

For the general case, $u$ is only convex in the contact set. 
Let $S_0=\{ x\in \O : \det D^2u =0\}.$ By Sard's Theorem (see \cite{EG} or \cite{F}) we have $|Du(S_0)|=0.$
Since $E=\{u=\G_u\}\cap \O$ is a Borel set, $E\cap S_0$ and $E\setminus S_0$ are also Borel sets.
Hence $$|Du(E)|=|Du(E\cap S_0)|+|Du(E\setminus S_0)|=|Du(E\setminus S_0)|$$
and by \eqref{eq1} and \eqref{eq2} we have
\[
\begin{split}
\left|D\G_u(E)\right|
&\leq \left|Du(E)\right|=\left|Du(E\setminus S_0)\right|
= \int_{E\setminus S_0}(\det D^2u) dx\\
&\leq C\int_{E\setminus S_0}\frac{(Lu)^2}{x_1^{2}}dx\leq C\int_{E\setminus S_0}{f^2}{x_1^{2}}dx\leq C\int_{E}{f^2}{x_1^{2}}dx.
\end{split}
\]
\end{proof}

As a corollary, we get the weak maximum principle for the operator
$L$ in \eqref{L}.
\begin{theorem}[Weak Maximum Principle]\label{WMP}
Let $\Omega$ be a bounded open set in $\R^2$ and let $u,v\in C^2(\Omega)\cap C(\bar \Omega)$  such that
$u\leq v$ on $\p \O$ and $Lv\leq Lu$ in $\O.$ Then, $u\leq v$ in $\O.$
\end{theorem}
\begin{proof}
Apply Theorem \ref{ABP} to $w=v-u.$
\end{proof}

An alternative  proof of the Weak Maximum Principle, without using Theorem \ref{ABP}, can be found in \cite[Corollary 1.3.]{L}.

\section{Grushin metric and sublevel sets}\label{metric}
In this section we recall definition and basic properties of the Grushin metric. 

The vector fields defined in \eqref{grushin}  induce on $\R^2$ a metric $d_{CC}$ in the following way (see \cite{FL}, 
\cite{FP} and \cite{NSW}, \cite{BLU}).
\begin{definition}[Carnot Carath\'eodory metric]\label{CC}
A Lipschitz continuous curve $\gamma: [0,T] \rightarrow \R^2,$  $T
\geq 0$, is subunit if there exists a vector of measurable functions
$h = (h_1,h_2) : [0,T] \rightarrow \R^2$
 such that
 $\gamma'(t)=\sum_{j=1}^2 h_j(t)X_j(\gamma(t))$ and $\sum_{j=1}^2 h^2_j(t)\leq 1$ for a.e. $t\in [0,T].$
 Define the Carnot Carath\'eodory distance $d_{CC} : \R^2 \times \R^2 \rightarrow [0, +\infty)$ by setting
\[
\begin{split}
d_{CC}(x,y) = \inf \{T \geq 0 : & \, \textrm{there exists a subunit curve} \, \gamma : [0,T] \rightarrow \R^2 \\
&\, \textrm{such that}\,
\gamma(0) = x \, \textrm{and} \, \gamma(T) = y\}.
\end{split}
\]
\end{definition}
It is well-known  that $d_{CC}(x,y)$ is finite for all $x,y,$ because the vector fields are smooth and satisfy H\"ormander condition (see  \cite{BLU}, \cite{NSW}).

The first structure Theorem  below, which is a special case of the results proved by Franchi and Lanconelli in \cite{FL}, plays a central role in our study of barriers.

We denote by
$B_{CC}(x,r) = \{y \in R^2 : d_{CC}(x,y) < r\}$ the balls in $\R^2$ defined by the metric $d_{CC}.$

For $j=1,2$ define the functions $F_j : \R^2 \times [0, +\infty) \rightarrow [0, +\infty)$ by
$$F_1(x, r) = r, \quad F_2(x, r) = r(|x_1| + F_1(x, r))=r(|x_1| +r).$$
Note that  $r \rightarrow F_j (x, r)$ is increasing and it satisfies the following doubling property
$$F_j(x,2r) \leq CF_j(x,r), x \in \R^2, 0 < r < \infty$$
 for all $j = 1,2.$

The structure of the balls $B_{CC}(x, r)$ can be described by means of the boxes
\[
\begin{split}
Box(x, r) :=& \{x + h : |h_j | < F_j (x, r), j = 1, 2\}\\
= &\{x + h : |h_1| < r, |h_2|<r|x_1|+r^2\}.
\end{split}
\]
For any fixed $x \in\R^2$ the function $F_j(x,\cdot)$
 is strictly increasing and maps $]0,+\infty[$ onto itself. We denote its inverse by $G_j(x,\cdot) = F_j(x,\cdot)^{-1}.$
 Precisely, $G_1(x,r)=r$ and $G_2(x,r)=\frac{-|x_1|+\sqrt{x_1^2+4r}}{2}.$
  The following structure theorem is proved in \cite{FL}
\begin{theorem}\label{FRLa}
There exists a constant $C > 0$ such that:
\[
Box(x,C^{-1}r) \subset B_{CC}(x,r) \subset Box(x,Cr), x \in \R^2, r \in ]0,+\infty[
\]
and
\[
C^{-1}d_{CC}(x,y)\leq \sum_{j=1}^2G_j(x,|y_j -x_j|)\leq Cd_{CC}(x,y), \quad x,y\in \R^2.
\]
\end{theorem}
Here 
\begin{equation}\label{F}
\sum_{j=1}^2G_j(x,|y_j -x_j|)=|y_1-x_1|+\frac{\sqrt{x_1^2+4|y_2-x_2|}-|x_1|}{2}.
\end{equation}
For $j=1,2$ define the real numbers $d_j$ by
$
 d_1=1, d_2=2.
$
 We say that $d_j$ is the degree of the variable $x_j .$ Note that $F_j (0, r) = r^{d_j} .$
 
 Moreover, there exists a group of dilations $(\delta_t)_{t>0}$,
 \begin{equation}\label{dil}
\delta_t: \R^2 \rightarrow \R^2, \quad
\delta_t(x)=(tx_1,t^2x_2)
\end{equation}
such that the vector fields $X_j$ are $\delta_t$-homogeneous of degree one, i.e.
for every $u\in C^1(\R^2),$
\begin{equation}\label{invariance}
X_j(u\circ \delta_t)(x) = t(X_ju)(\delta_t(x)),
\end{equation}
for all
$x \in \R^2$ and $t>0.$

The number $Q =d_1+d_2= 3$ is the homogeneous dimension of $\R^2$ with respect to $\delta_t.$
 The size of balls
 in the metric $d_{CC}$ has been described by Franchi and Lanconelli \cite{FL} by means of the boxes
$Box(x, r).$ Precisely, there exists a positive constant $C>0$ such that
\begin{equation}\label{homog}
C^{-1}
F(x,r)
 \leq |B_{CC}(x,r)|\leq  C
 F(x,r)
 \end{equation}
 for every $x\in \R^2$ and $r>0.$
 Here $ |B|$ is the Lebesgue measure of $B$ and $F(x,r)=F_1(x,r)\cdot F_2(x,r)=r^2(|x_1| +r).$

 This means that
 the measure of  balls with radius $r$ and  center at $x$ with   $r\leq |x_1|$ is of Euclidean type $|x_1|r^2,$ whereas the Lebesgue measure of $B_{CC}(x,r)$ with $|x_1| < r$ is comparable to $r^{3}$. Unfortunately, Theorem \ref{FRLa} and \eqref{homog} are not enough to conclude that  the ring condition in \cite[Definition 2.6]{DGL} holds true in $B_{CC}.$
 
 Inspired by \eqref{F}, we then define a new quasi distance in $\R^2$ as
 \[
 \tilde d(x,y)=|x_1-y_1|+ \sqrt{x_1^2+y_1^2+4|x_2-y_2|}-\sqrt{x_1^2+y_1^2},
 \]
 where $x=(x_1,x_2), y=(y_1,y_2).$
 Obviously, $\tilde d$ is $1/2$ -H\"older continuous.
 
 We denote by
 \begin{equation}\label{ball}
B(x,r) = \{y \in R^2 : \tilde d(x,y) < r\}
\end{equation}
 the balls in $\R^2$ defined by the  quasi metric $\tilde d.$

\begin{theorem}[I Structure Theorem]\label{charact}
There exists a constant $C > 0$ such that:
\[
Box(y,C^{-1}r) \subset B(y,r) \subset Box(y,Cr), y \in \R^2, r \in ]0,+\infty[
\]
\end{theorem}

 \begin{proof}
 If $x\in B(y,r)$ then $|x_1-y_1|<r$ and $\sqrt{x_1^2+y_1^2+4|x_2-y_2|}<r+\sqrt{x_1^2+y_1^2}.$
 In particular, 
\[
\begin{split}
 4|x_2-y_2| &<(r+\sqrt{x_1^2+y_1^2})^2-(x_1^2+y_1^2)=r^2+2r\sqrt{x_1^2+y_1^2}\\
 & \leq r^2+2r\sqrt{(|y_1|+r)^2+y_1^2}\leq r^2+2r(2|y_1|+r)=3r^2+4r|y_1|.
 \end{split}
\]
Hence, $x\in Box (y,r).$

Vice versa, if $x\in Box (y,r),$ then $|x_1-y_1|<r$ and $|x_2-y_2|<r(r+|y_1|).$ We have
\[
\begin{split}
|x_1-y_1|+ \sqrt{x_1^2+y_1^2+4|x_2-y_2|}-\sqrt{x_1^2+y_1^2} &<r+ \sqrt{x_1^2+y_1^2+4r(r+|y_1|)}-\sqrt{x_1^2+y_1^2}\\
&=r+\sqrt{x_1^2+(|y_1|+2r)^2}-\sqrt{x_1^2+y_1^2}\\
&\leq r+2r.
 \end{split}
\]
Hence, $x\in B(y,3r).$
 \end{proof}

By Theorem \ref{charact} it is easy to recognize that
there exists a positive constant $C>0$ such that
\begin{equation}\label{homog2}
C^{-1}
r^2(r+|y_1|)
 \leq |B(y,r)|\leq  C
r^2(r+|y_1|)
 \end{equation}
 for every $y\in \R^2$ and $r>0.$

Moreover, there exists a positive constant $C_D$ such that the following Doubling Property holds true
\begin{equation}\label{dconstant}
 0<|B(x,2r)|\leq C_D|B(x,r)|, \quad \forall x \in \R^2, \forall r>0.
\end{equation}
 
The quasi metric balls $B$ satisfy the following Ring Condition, which will be crucial in Section \ref{dbp}.
 \begin{theorem}[Ring Condition]\label{ringc}
 There exists a nonnegative function $\omega,$ such that 
 \begin{itemize}
 \item
$\left|B(x,r)\setminus B(x,(1-\e)r)\right|\leq \omega(\e)|B(x,r)|,$ for every ball $B(x,r)$ and for all small $\e>0.$ 
 \item
 $\omega(\e)=O(\epsilon)$ as $\e \rightarrow 0^+$
\end{itemize}
\end{theorem}
\begin{proof}
For every fixed $y\in \R^2$ and for every $r>0,$ by Fubini's Theorem  we have 
\[
\begin{split}
f(r):=&\int_{B(y,r)}dx=\int_{y_1-r}^{y_1+r}\left(\int_{4|x_2-y_2|<\left(r-|x_1-y_1|+\sqrt{x_1^2+y_1^2}\right)^2-(x_1^2+y_1^2)}dx_2
\right)dx_1\\
=&\frac{1}{2}\int_{-r}^r\left((r-|t|)^2+2(r-|t|)\sqrt{(y_1+t)^2+y_1^2}\right)dt.
\end{split}
\]
Remark that the function $f$ is differentiable  at any point and 
\[
\begin{split}
f'(r)
= &\int_{-r}^r\left((r-|t|)+\sqrt{(y_1+t)^2+y_1^2}\right)dt=r^2+\int_{-r}^r\sqrt{(y_1+t)^2+y_1^2} \,dt,\\
0\leq f'(r)\leq &r^2+4|y_1|r+r^2\leq 4r(r+|y_1|).
\end{split}
\]

In particular, by Lagrange mean value Theorem, we have that there exists $\theta\in (1-\e,1)$ such that
\[
{\left|B(y,r)\setminus B(y,(1-\e)r)\right|}=f(r)-f((1-\e)r)=f'(\theta r) \e r\leq 4\e r^2(r+|y_1|)\leq c\e |B(y,r)|,
\]
where $c$ is a positive universal constant because of \eqref{homog2}.
\end{proof}
 In the sequel we will also need the following characterization of $Box(y,r)$ and consequently of
 $B(y,r)$ in term of the function
\begin{equation}\label{rho}
\rho({x},{y})= \left((x_1^2-y_1^2)^2+4(x_2-y_2)^2\right)^{1/4}.
\end{equation}

 It is well known that the function $\G(x,0):=\rho^{2-Q}(x,0)$ is the fundamental solution
with pole at the origin of the subelliptic Laplacian $X_1^2+X_2^2,$
where $Q=3$ is the homogeneous dimension.

For every fixed $r>0$ and $y=(y_1,y_2)\in \R^2$  define
 \begin{equation}\label{tg}
\tilde g_r(x,y)=\left\{
\begin{array}{cc}
  \rho(x,y), &   \textrm{if }\, |y_1|<r   \\
 \frac{1}{|y_1|}\rho^2(x,y), &  \textrm{if }\, |y_1|\geq r \\
  \end{array}
 \right.
 \end{equation}
 \begin{remark}\label{rem}
 The function $\tilde g_r(x,y)$ has two zeros at $(y_1,y_2)$ and at $(-y_1,y_2)$ and it is an even function with respect to the first variable $x_1.$
 \end{remark}
 In the sequel we will study the sublevel sets of the function $\tilde g_r(\cdot,y)$
 \[
 \tilde G(y,r):=\{
 x\in \R^2: \tilde g_r(x,y)<r
 \}.
 \]

In order to avoid two zeros, 
 for every fixed $r>0$ we also define
 \begin{equation}\label{g}
 g_r(x,y)=\left\{
\begin{array}{cc}
  \rho(x,y), &   \textrm{if }\, |y_1|<r   \\
 \frac{1}{|y_1|}\rho^2(x,y), &  \textrm{if }\, |y_1|\geq r \, \textrm{and }\, x_1 y_1\geq 0\\
+ \infty, &  \textrm{if }\, |y_1|\geq r \, \textrm{and }\, x_1 y_1<
0
  \end{array}
 \right.
 \end{equation}
 In the following Theorem, which is the main result of this section, we compare the sublevel sets of the function $g_r(\cdot,y)$
 \[
 G(y,r):=\{
 x\in \R^2: g_r(x,y)<r
 \}
 \]
  with the $Box(y,r).$
 \begin{theorem}[II Structure Theorem]\label{equiv}
 There exists a constant $C > 0$ such that:
\begin{equation}\label{inclusion}
   Box(y,C^{-1}r) \subset G(y,r) \subset Box(y,Cr), y \in \R^2, r \in ]0,+\infty[
 \end{equation}
\end{theorem}
\begin{proof}
To prove the first inclusion in \eqref{inclusion}, take $x\in
Box(y,r/5)$ and assume $|y_1|<r,$ then
\[
|x_1-y_1|<r/5 , \quad |x_2-y_2|<r/5(r/5+|y_1|)<1/5(1/5+1)r^2
\]
and we get
\[
\begin{split}
g_r^4(x,y)&=\rho^4(x,y)=(x_1-y_1)^2(x_1+y_1)^2+4(x_2-y_2)^2\\
&<(r/5)^2(|x_1|+|y_1|)^2+4r^4(1/5)^2(1/5+1)^2\\
&\leq (r/5)^2(|x_1-y_1|+2|y_1|)^2+4r^4(1/5)^2(1/5+1)^2\\
&< r^4(1/5)^2(1/5+2)^2+4r^4(1/5)^2(1/5+1)^2\\
&<r^4(1/5)(1/5+2)^2=r^4 (121/125)<r^4
\end{split}
\]
If $x\in Box(y,r/5)$ and $|y_1|\geq r,$ then
\[
|x_1-y_1|<r/5 , \quad |x_2-y_2|<r/5(r/5+|y_1|)\leq 1/5(1/5+1)r |y_1|
\]
and we get $x_1y_1>0$ and
\[
\begin{split}
g_r^2(x,y)|y_1|^2&=\rho^4(x,y)=(x_1-y_1)^2(x_1+y_1)^2+4(x_2-y_2)^2\\
&<(r/5)^2(|x_1|+|y_1|)^2+4r^2|y_1|^2(1/5)^2(1/5+1)^2\\
&\leq (r/5)^2(|x_1-y_1|+2|y_1|)^2+4r^2|y_1|^2(1/5)^2(1/5+1)^2\\
&<r^2|y_1|^2(1/5)^2(1/5+2)^2+4r^2|y_1|^2(1/5)^2(1/5+1)^2\\
&<r^2|y_1|^2(1/5)(1/5+2)^2=r^2|y_1|^2(121/125)<r^2|y_1|^2\end{split}
\]
In particular, we have proved that
\[
Box(y,r/5) \subset G(y,r).
\]
To prove the second inclusion in \eqref{inclusion} take $x\in
G(y,r)$ and  assume $|y_1|<r,$ then
\[
|x_1^2-y_1^2|<r^2, \quad 4|x_2-y_2|^2<r^2.
\]
In particular
\[
|x_2-y_2|<\frac{r^2}{2}<r(r+|y_1|),
\]
and
\begin{equation}\label{star}
r^2>|x_1-y_1||x_1+y_1|\geq |x_1-y_1|\left| |x_1-y_1|-2|y_1| \right|.
\end{equation}
We have two cases
\begin{itemize}
  \item If $|x_1-y_1|\leq 2|y_1|,$ then $x\in Box(y,2r).$
  \item If $|x_1-y_1|>2|y_1|,$ then by \eqref{star}
  \[
|x_1-y_1|<\sqrt{|y_1|^2+r^2}+|y_1|<3r
  \]
  and $x\in Box(y,3r).$
\end{itemize}
If $x\in G(y,r)$ and $|y_1|\geq r$ then $x_1y_1\geq 0$
\[
|x_1^2-y_1^2|<r|y_1|, \quad 4|x_2-y_2|^2<r^2|y_1|^2.
\]
Thus, we have
\[
|x_2-y_2|<\frac{r|y_1|}{2}<r(r+|y_1|),
\]
and, since $x_1y_1\geq 0$
\[
r|y_1|>|x_1-y_1||x_1+y_1|\geq |x_1-y_1||y_1|.
\]
Hence $x\in Box(y,r).$

In particular, we have proved that
\[
G(y,r)\subset Box(y,3r).
\]
\end{proof}

\section{Barriers}\label{sbarrier}
Let $\rho$ be the function in \eqref{rho}.
\begin{lemma}\label{Lbarrier}
Let $y=(y_1,0)\in \R^2$ be fixed and consider the function $\phi(x)=\rho^\a(x,y).$ For $\a\leq 2-3 \L/\l,$ $\phi$ is a
classical solution of $L\phi\geq 0$ in the set $\{ \rho>0\}.$
\end{lemma}
\begin{proof}
By
a simple calculation we get
\begin{equation}\label{LphiII}
\begin{split}
L\phi =& a_{11} \p_{x_1}(\a \rho^{\a-1}\rho_{x_1})+2a_{12}{x_1}\p_{x_2}(\a \rho^{\a-1}\rho_{x_1})+a_{22}x_1^2\p_{x_2}(\a \rho^{\a-1}\rho_{x_2})\\
=& a_{11}\left( \a(\a-1)\rho^{\a-2}\rho_{x_1}^2+\a
\rho^{\a-1}\rho_{{x_1}{x_1}}\right)
+2a_{12}{x_1}\left( \a(\a-1)\rho^{\a-2}\rho_{x_1}\rho_{x_2}+\a \rho^{\a-1}\rho_{{x_1}{x_2}}\right)\\
&+a_{22}x_1^2\left( \a(\a-1)\rho^{\a-2}\rho_{x_2}^2+\a
\rho^{\a-1}\rho_{{x_2}{x_2}}\right)
\end{split}
\end{equation}
and
\begin{equation}\label{rhoII}
\begin{split}
\rho_{x_1}=& \rho^{-3}(x_1^2-y_1^2)x_1, \\
\rho_{x_2}=& 2\rho^{-3}{x_2},\\
\rho_{{x_1}{x_1}}=&(3x_1^2-y_1^2)\rho^{-3}-3\rho^{-7}x_1^2(x_1^2-y_1^2)^2= \rho^{-7}(12x_2^2x_1^2-y_1^2\rho^4)\\
\rho_{{x_1}{x_2}}=&\rho^{-7}(-6{x_2}x_1(x_1^2-y_1^2))\\
\rho_{{x_2}{x_2}}=&2\rho^{-3}-12\rho^{-7}x_2^2=\rho^{-7}(2(x_1^2-y_1^2)^2-4x_2^2)=\rho^{-7}(3(x_1^2-y_1^2)^2-\rho^4)\\
\end{split}
\end{equation}
By substituting \eqref{rhoII} in \eqref{LphiII} and by taking into
account \eqref{ell}  and that $\a\leq -1$ is negative and $a_{11}, a_{22}$ are
positive,  we get
\begin{equation}\label{Lphi2II}
\begin{split}
L\phi =&\a (\a-1) x_1^2 \rho^{\a-8}\left(
a_{11} (x_1^2-y_1^2)^2+4a_{12}{x_2}(x_1^2-y_1^2)+4a_{22}x_2^2 \right)\\
& +\alpha \rho^{\a-8}\left(
a_{11} (12 x_1^2 x_2^2-y_1^2\rho^4)
 -12a_{12}{x_2}x_1^2(x_1^2-y_1^2)
 +a_{22}x_1^2(3(x_1^2-y_1^2)^2-\rho^4)
\right)\\
\geq & \a (\a-1) x_1^2 \rho^{\a-8}\left(
\l \rho^4 \right) -\alpha a_{11} \rho^{\a-4}y_1^2
 \\
& +\alpha \rho^{\a-8}\left(
 a_{11} (12 x_1^2 x_2^2)
 -12a_{12}{x_2}x_1^2(x_1^2-y_1^2)
 +a_{22}x_1^2(3(x_1^2-y_1^2)^2-\rho^4)
\right)\\
\geq & -\alpha a_{11} \rho^{\a-4}y_1^2+ \a x_1^2 \rho^{\a-4}\left((\a-2)  \l +3\L \right)\geq -\alpha a_{11} \rho^{\a-4}y_1^2 \geq 0.
\end{split}
\end{equation}
\end{proof}

\begin{lemma}\label{lemma4.1}
There exists positive universal constants $C>0$ and $M>1$ such that for every $y=(y_1,0)\in \R^2$ and $r\in \R^+$ there is a $C^2$ function $\tilde \varphi: \R^2 \rightarrow \R$ such that
\begin{equation}\label{4.1}
\tilde \varphi \geq 0,\,  \textrm{on} \, \R^2 \setminus \tilde G(y,2r),
\end{equation}
\begin{equation}\label{4.2}
\tilde \varphi\leq -2, \,  \textrm{in} \,  \, \tilde G(y,r)
\end{equation}

\begin{equation}\label{4.3}
\tilde \varphi \geq -M, \quad L\tilde \varphi(x)\leq C \frac{x_1^2}{r^2(r+|y_1|)^2}\,\zeta(x), \,
\textrm{in} \, \R^2
\end{equation}
where $0\leq \zeta\leq 1$ is a continuous function in $\R^2$ with $supp \, \zeta \subset \overline{\tilde G(y,r)}.$

\end{lemma}
\begin{proof}
Consider the constant $\alpha={ 2-3 \L/\l}$ and recall that $\alpha \leq -1.$
We define
\[
\varphi(x) = M_1 - M_2 \rho^\alpha(x,y), \,\,  \textrm{in} \, \R^2 \setminus \{x: \rho(x,y)=0\}.
\]
We  choose $M_1,M_2$ such that
\[
\varphi |_{\p \tilde G(y,2r)} =0, \quad \varphi |_{\p \tilde G(y,r)} = -2
\]
Then  take $-m=\varphi|_{\p \tilde G(y,r/2)}.$
We now show that $M_1, M_2, m$ are positive. Indeed, if $|y_1|\geq 2r$ then 
\begin{equation}\label{c}
M_1=\frac{ 2 \cdot 2^{\alpha/2}}{1-2^{\alpha/2}}, \quad M_2=\frac{2}{(r|y_1|)^{\alpha/2}(1-2^{\alpha/2})}
\end{equation}
and
\[
-m=\frac{2}{1-2^{\alpha/2}}\left( 
2^{\alpha/2}-\frac{1}{2^{\alpha/2}}
\right).
\]
If $|y_1|<r/2$ then 
\begin{equation}\label{c2}
M_1=\frac{ 2 \cdot 2^{\alpha}}{1-2^{\alpha}}, \quad M_2=\frac{2}{r^{\alpha}(1-2^{\alpha})}
\end{equation}
and
\[
-m=\frac{2}{1-2^{\alpha}}\left( 
2^{\alpha}-\frac{1}{2^{\alpha}}
\right).
\]
If $r/2\leq |y_1|<r$ then 
\begin{equation}\label{c3}
M_1=\frac{ 2 \cdot 2^{\alpha}}{1-2^{\alpha}}, \quad M_2=\frac{2}{r^{\alpha}(1-2^{\alpha})}
\end{equation}
and
\[
-m=\frac{2}{1-2^{\alpha}}\left( 
2^{\alpha}-\left(\frac{|y_1|}{2r}\right)^{\alpha/2}
\right)
>\frac{2}{1-2^{\alpha}}\left( 
2^{\alpha}-\frac{1}{2^{\alpha/2}}
\right).
\]

If $r\leq |y_1|<2r$ then 
\begin{equation}\label{c4}
M_1=\frac{ 2 \cdot 2^{\alpha} r^{\alpha/2}}{|y_1|^{\alpha/2}-(4r)^{\alpha/2}}<
\frac{ 2 \cdot 2^{\alpha}}{2^{\alpha/2}-4^{\alpha/2}}
, \quad M_2=\frac{2}{r^{\alpha/2}(|y_1|^{\alpha/2}-(4r)^{\alpha/2})}
\end{equation}
and
\[
-m=-2\frac{((|y_1|/2)^{\alpha/2}-(4r)^{\alpha/2})}{(|y_1|^{\alpha/2}-(4r)^{\alpha/2})}
>
-2\frac{\left((1/2)^{\alpha/2}-4^{\alpha/2}\right)}{(2^{\alpha/2}-4^{\alpha/2})}
\]

Therefore $M_1$ and $m$ are uniformly bounded, while by \eqref{c}, \eqref{c2}, \eqref{c3}, \eqref{c4} we have
\begin{equation}\label{M2}
\frac{\tilde C}{r^{\alpha/2}(r+|y_1|)^{\alpha/2}}\leq M_2\leq \frac{C}{r^{\alpha/2}(r+|y_1|)^{\alpha/2}}.
\end{equation}
Define 
\[
h(t)=\left\{
\begin{array}{cc}
  t,  & \,\textrm{if} \, -m \leq t  \\
-\int_{t+m}^0\frac{1}{1+s^{2\beta}}ds-m,   &  \,\textrm{if} \, \,  t< -m
   \end{array}
   \right.
\]
where $\beta \in \N$ will be chosen in a moment.
For $t<-m$ we have
\[
h'(t)=\frac{1}{1+(t+m)^{2\beta}}>0, \quad h''(t)=- \frac{2\beta(t+m)^{2\beta -1}}{(1+(t+m)^{2\beta})^{2}},
\]
therefore $h\in C^2(\R).$

Define
\[
\tilde \varphi(x) =
\left\{
\begin{array}{cc}
  h(\varphi(x)),  & \,\textrm{if} \, \rho(x,y)>0 \\
 -\int_{-\infty}^0\frac{1}{1+s^{2\beta}}ds-m,   &  \,\textrm{if} \, \, \rho(x,y)=0
   \end{array}
   \right. 
\]
Remark that $\tilde \varphi$  satisfies \eqref{4.1}, \eqref{4.2}.  
By recalling \eqref{rhoII}, \eqref{Lphi2II}, \eqref{M2} and by choosing $\beta \in \N$ and $2\beta> \max \{ 1, 1-4/\alpha\},$ in ${\tilde G(y,r/2)}$ we have

\[
\begin{split}
L\tilde \varphi
&=h''(\varphi)\left(a_{11}\varphi_{x_1}^2+2a_{12}x_1\varphi_{x_1}\varphi_{x_2}+a_{22}x_1^2\varphi_{x_2}^2\right)+h'(\varphi)
L\varphi\\
&\leq |h''(\varphi)|\L( (X_1\varphi)^2+(X_2\varphi)^2)\\
&\leq |h''(\varphi)|\L \tilde C M_2^2\, {x_1^2}{\r^{2\alpha -4}}
\leq CM_2^{1-2\beta}\r^{\alpha -4-2\alpha \beta}x_1^2
\\
&\leq \frac{C}{r^2(r+|y_1|)^2} x_1^2
\end{split}
\]
because $\rho<Cr^{1/2}(r+|y_1|)^{1/2}$ in ${\tilde G(y,r/2)}.$

Finally, by Lemma \ref{Lbarrier}, $L\tilde \varphi =L\varphi \leq 0$ in
$\R^2\setminus \tilde G(y,r/2),$ 
and $L\tilde \varphi \leq C \frac{x_1^2}{r^2(r+|y_1|)^2}$ in
$\overline{G(y,r/2)}.$ We now take a continuous function $\zeta$ such 
that $\zeta \equiv 1$ in $\overline{G(y,r/2)},$ $\zeta \equiv 0$
outside $\tilde G(y,2r/3 ),$  $0\leq \zeta \leq 1,$  $\zeta$ is even in the $x_1$ variable and such that \eqref{4.3} 
holds true.
\end{proof}
\section{Critical Density}\label{CDE}
By combining Theorem \ref{ABP} with the geometry of the sets $G$ and the barrier in Lemma \ref{lemma4.1}
 we get a first rough critical density estimate. 
 
 Recall that by definition $G(y,r)=\tilde G(y,r) \cap \{(x_1,x_2): \, x_1 y_1 \geq 0\}$ for $|y_1|\geq r,$ while
$G(y,r)=\tilde G(y,r)$ for $|y_1|<r$
 and
 that
 by Theorem \ref{equiv} and Theorem \ref{charact} $B(y,r/c)\subset G(y,r)\subset B(y,cr)$ for a universal constant $c>0.$
\begin{theorem}\label{CDG0}
There exist universal constants $0<\nu<1$ and $M>1$ such that, for every $y=(y_1,0)\in\R^2$ and $r>0,$
if $u\in C^2(G(y,2r))\cap C(\overline{G(y,2r)})$ is a nonnegative solution of $Lu\leq 0$ in $G(y,2r)$ and $\inf_{G(y,r)} u\leq 1$
then
\begin{equation}\label{cde}
\left|\left\{
u\leq M
\right\}\cap G(y,3r/2)\right|\geq \frac{\nu}{\max\left\{r+|y_1|, \frac{1}{r+|y_1|}\right\}}|G(y,3r/2)|.
\end{equation}
Here $|\cdot|$ is the Lebesgue measure.
 \end{theorem}
 \begin{proof}
Take $\tilde \varphi$ as in Lemma \ref{lemma4.1} and $w=u+\tilde \varphi.$ We
have $Lw \leq C\frac{x_1^2\zeta}{r^2(r+|y_1|)^2}$ in  $G(y,2r)$ and $w\geq 0$ on $\p G(y,2r)$
by \eqref{4.1}. Moreover, $\inf_{G(y,r)}w\leq -1$ by \eqref{4.2}.

We
next apply Theorem \ref{ABP}   to $w$ in $G(y,2r).$
By Theorem \ref{charact} and Theorem \ref{equiv} we have 
$\diam(G(y,2r))\leq C\max\{ r, r(r+|y_1|)\},$ $c^{-1} r^2(r+|y_1|) \leq |G(y,r)|\leq c r^2(r+|y_1|)$ and $\sup_{G(y,2r)}|x_1|\leq \tilde c (r+ |y_1|),$
for some positive universal constants $C,c,\tilde c.$
By recalling
that $0\leq \zeta \leq 1,$ and $supp \, \zeta \subset \tilde G(y,r)$
 we get
\begin{equation}\label{anna}
\begin{split}
1 & \leq  C \diam(G(y,2r)) \left( \int_{\{w=\Gamma_w \} \cap G(y,2r) }\frac{(x_1 C \zeta)^2}{r^4(r+|y_1|)^4}dx
\right)^{1/2} \\
 &\leq  C \frac{ \max\{ r, r(r+|y_1|)\}}{r^2 (r+|y_1|)}
 \left( \int_{\{w=\Gamma_w \} \cap G(y,2r) \cap \tilde G(y,r) }{\zeta^2}dx
\right)^{1/2} \\
&\leq \tilde C  \frac{ \max\{ r, r(r+|y_1|)\}}{r^2 (r+|y_1|)}
\left| \{w=\Gamma_w \}  \cap G(y,2r) \cap \tilde G(y,r)
\right|^{1/2}\\
&\leq \tilde C  \frac{ \max\{ r, r(r+|y_1|)\}}{r^2 (r+|y_1|)}\left| \{u\leq M \}  \cap G(y,2r) \cap \tilde G(y,r) \right|^{1/2}\\
& \leq \tilde C c^{1/2}   \frac{ \max\{ r, r(r+|y_1|)\}}{r^2 (r+|y_1|)}
\frac{(r^2(r+|y_1|))^{1/2}}{|G(y,r)|^{1/2}}\left| \{u\leq M \}  \cap G(y,2r) \cap \tilde G(y,r) \right|^{1/2}\\
& \leq \tilde C c^{1/2}   \frac{ \max\{ 1, (r+|y_1|)\}}{(r+|y_1|)^{1/2}}
\frac{\left| \{u\leq M \}  \cap G(y,2r) \cap \tilde G(y,r) \right|^{1/2}}{|G(y,r)|^{1/2}}
\end{split}
\end{equation}
because $w=\Gamma_w$ implies $w(x)\leq 0$ and thus $u(x)\leq -\varphi(x)\leq M.$

Moreover,
\[
\begin{split}
1 & \leq C \frac{ \max\{ 1, (r+|y_1|)\}}{(r+|y_1|)^{1/2}}
\frac{\left| \{u\leq M \}  \cap G(y,2r) \cap \tilde G(y,3r/2) \right|^{1/2}}{|G(y,3r/2)|^{1/2}}
\end{split}
\]

We now remark that, for $|y_1|<3r/2$ or for $|y_1|\geq 2r,$ we have $G(y,2r) \cap \tilde G(y,3r/2)=G(y,3r/2)$
and the estimate \eqref{cde} follows. 

On the contrary, for  $(3/2)r\leq |y_1|< 2r$ we have $G(y,2r) \cap \tilde G(y,3r/2)=G(y,3r/2)\cup G(-y,3r/2)$ with $G(y,3r/2)\cap G(-y,3r/2)$ empty.
The idea is to apply estimate \eqref{anna} in a smaller ball.
Let us call $r=(3/2)R.$ We obviously have $R=(2/3)r$ and $2R=(4/3)r<2r.$ Hence, $u\in C^2(G(y,2R))\cap C(\overline{G(y,2R)})$ is a nonnegative solution
of $Lu\leq 0$ in $G(y,2R).$ If $\inf_{G(y,r)}u=\inf_{G(y,(3/2)R)}\leq 1$ then by arguing as in estimate \eqref{anna} we have
\[
\begin{split}
1 & \leq C \frac{ \max\{ 1, ((3/2)R+|y_1|)\}}{((3/2)R+|y_1|)^{1/2}}
\frac{\left| \{u\leq M \}  \cap G(y,2R) \cap \tilde G(y,(3/2)R) \right|^{1/2}}{|G(y,(3/2)R)|^{1/2}}\\
&=C \frac{ \max\{ 1, ((3/2)R+|y_1|)\}}{((3/2)R+|y_1|)^{1/2}}
\frac{\left| \{u\leq M \}  \cap G(y,(3/2)R) \right|^{1/2}}{|G(y,(3/2)R)|^{1/2}}
\end{split}
\]
because $|y_1|\geq (3/2)r=(9/4)R>2R.$ By recalling that  $r=(3/2)R,$ we have
\[
\begin{split}
1 &
\leq  C \frac{ \max\{ 1, (r+|y_1|)\}}{(r+|y_1|)^{1/2}}
\frac{\left| \{u\leq M \}  \cap  G(y,r) \right|^{1/2}}{|G(y,r)|^{1/2}}\\
&\leq
\tilde  C \frac{ \max\{ 1, (r+|y_1|)\}}{(r+|y_1|)^{1/2}}
\frac{\left| \{u\leq M \}  \cap  G(y,3r/2) \right|^{1/2}}{|G(y,3r/2)|^{1/2}}.\\
\end{split}
\]

\end{proof}

We now perform the constant in Theorem \ref{CDG0} by taking into account the geometry of the problem.
\begin{theorem}[Critical Density]\label{CDG}
There exist universal constants $0<\nu<1$ and $M>1$ such that, for every $y=(y_1,y_2)\in\R^2$ and $r>0,$
if $u\in C^2(G(y,2r))\cap C(\overline{G(y,2r)})$ is a nonnegative solution of $Lu\leq 0$ in $G(y,2r)$ and $\inf_{G(y,r)} u\leq 1$
then
\[
\left|\left\{
u\leq M
\right\}\cap G(y,3r/2)\right|\geq {\nu}|G(y,3r/2)|.
\]
Here $|\cdot|$ is the Lebesgue measure.
 \end{theorem}
 \begin{proof}
The strategy is to combine Theorem \ref{CDG0} with dilations and translations and it has been inspired us by Ermanno Lanconelli during a private conversation.

\noindent
I STEP. Apply Theorem \ref{CDG0} for $y_1\in [-1,1]$ and $r=1.$
We then have that \eqref{cde}  holds true with a universal positive constant and precisely
\begin{equation}\label{cde1}
\left|\left\{
u\leq M
\right\}\cap G((y_1,0),3/2)\right|\geq {\nu}|G((y_1,0),3/2)|.
\end{equation}

\noindent
II STEP.
Assume $|y_1|\leq r.$
We change variables and we recall \eqref{invariance}
 and \eqref{homog}, together with Theorem \ref{charact} and Theorem \ref{equiv}.
 We introduce a change of variable that preserves the equation: fix $y_2\in \R$ and $r>0$ and let
 \begin{equation}\label{T}
 T(x)=T(x_1,x_2)=(rx_1, y_2+r^2x_2).
 \end{equation}
 Remark that $T(x) \in G(y,r)$ iff $x\in G((y_1/r,0),1).$
We define $\tilde u(x)=u(T(x)).$ We have, for all $i\leq j\in \{1,2\}$
\[
\begin{split}
X_i\tilde u(x)=&r X_iu(T(x)), \\
X_jX_i\tilde u(x) =& r^2X_jX_i u(T(x)).
\end{split}
\]
Set $\tilde a_{11}(x)=a_{11}(T(x)), \tilde a_{12}(x)=a_{12}(T(x)), \tilde a_{22}(x)=a_{22}(T(x)), $ and 
$\tilde L=\tilde a_{11}X_1^2+2\tilde a_{12}
X_2X_1+\tilde a_{22}X_2^2,$ then
$\tilde L \tilde u(x)=r^2Lu(T(x)).$
We have that $\tilde u$
 satisfies the hypothesis of the theorem with  $y_1/r\in [-1,1]$ and by \eqref{cde1} in the first step we get
 \[
\left|\left\{
\tilde u\leq M
\right\}\cap G((y_1/r,0),3/2)\right|\geq {\nu}|G((y_1/r,0),3/2)|
\]
and
by \eqref{homog}, together with Theorem \ref{charact} and Theorem \ref{equiv},
 we conclude that for $|y_1|\leq r$ we have
\begin{equation}\label{rbig}
 \begin{split}
 \left|\left(\left\{
 u\leq M
\right\}\cap G(y,3r/2)\right)\right|=&
\left|T\left(\left\{
\tilde u\leq M
\right\}\cap G((y_1/r,0),3/2)\right)\right| \\
=&Cr^3\left|\left\{
\tilde u\leq M
\right\}\cap G((y_1/r,0),3/2)\right|\\
\geq &{\nu} Cr^3G((y_1/r,0),3/2)|= {\tilde\nu}|G(y,3r/2)|.
\end{split}
\end{equation}
In particular, \eqref{rbig} holds true at $y=(0,y_2)$ for every $r>0$ and for every $y_2.$

\noindent
III STEP. Fix a point $y=(y_1,0)$ such that $|y_1|=1.$ For $r\geq 1$ we have that \eqref{rbig} holds true.
For $0<r<1$ we apply again estimate \eqref{cde} and, by taking into account that 
${\max\left\{r+|y_1|, \frac{1}{r+|y_1|}\right\}}=r+1<2,$ we get the theorem for every $r.$

\noindent
IV STEP. Take an arbitrary point $y=(y_1,y_2)\in \R^2$ with $y_1\neq 0$ and an arbitrary $\tilde r>0$ and apply the dilation in \eqref{T} with $r=|y_1|.$
By the third step we have
\[
 \begin{split}
 \left|\left(\left\{
 u\leq M
\right\}\cap G(y,3\tilde r/2)\right)\right|=&
\left|T\left(\left\{
\tilde u\leq M
\right\}\cap G((y_1/|y_1|,0),3\tilde r/(2|y_1|))\right)\right| \\
=&C|y_1|^3\left|\left\{
\tilde u\leq M
\right\}\cap G((y_1/|y_1|,0),3\tilde r /(2|y_1|))\right|\\
\geq &{\nu} C|y_1|^3G((y_1/|y_1|,0),3\tilde r /(2|y_1|))|= {\tilde\nu}|G(y,3\tilde r/2)|.
\end{split}
\]
 \end{proof}

 \begin{theorem}\label{cdB}
There exist universal constants  $\eta>2,$ $0<\nu, \theta<1$ and $M>1$ such that, for all $y\in \R^2$ and $r>0,$
if $u\in C^2(B(y,\eta r))\cap C(\overline{B(y,\eta r)})$ is a nonnegative solution of $Lu\leq 0$ in $B(y,\eta r)$ and
$\inf_{B(y,\theta r)} u\leq 1$
then
\[
\left|\left\{
u\leq M
\right\}\cap B(y,r)\right|\geq \nu|B(y,r)|
\]
 \end{theorem}
\begin{proof}
By Theorem \ref{charact} and Theorem \ref{equiv} we can substitute the sets $G$ in Theorem \ref{CDG} with the quasi metric balls $B,$ after rescaling.
\end{proof}

By arguing as in Theorem \ref{CDG} and by taking into account \eqref{anna} we get
\begin{corollary}\label{corpre}
There exist universal constants  $0<\nu<1$ and $M>1$ such that, for all $y\in \R^2$ and $r>0,$
if $u\in C^2(G(y, 2 r))\cap C(\overline{G(y,2 r)})$ is a nonnegative solution of $Lu\leq 0$ in $G(y,2 r)$ and
$\inf_{\tilde G(y,r)\cap G(y,2r)} u\leq 1$ then
\[
\left|\left\{
u\leq M
\right\}\cap \tilde G(y,r)\cap G(y,2r)\right|>  \nu|\tilde G(y,r)\cap G(y,2r) |.
\]
\end{corollary}
\begin{remark}
 Consider 
the symmetry ${S}$ with respect to the $x_2$ axis 
\begin{equation}\label{S}
{S}(x_1,x_2)=(-x_1,x_2).
\end{equation}
We first show that $S$ preserves the equation.
Let us call $u_{S}(x)=u(S(x))$ and $L_S=a_{11}(S(x))X_1^2+a_{12}(S(x))X_2X_1+a_{22}(S(x))X_2^2.$
We have that the coefficients of $L_S$ satisfy \eqref{ell} and $L_Su_S(x)=(Lu)(S(x)).$

By Corollary \ref{corpre},  
there exist universal constants  $0<\nu<1$ and $M>1$ such that, for all $y\in \R^2$ and $r>0,$
  if $u\in C^2(G(y, 2 r))\cap C(\overline{G(y,2 r)})$ is a nonnegative solution of $Lu\leq 0$ in $G(y,2 r)$ we have:
\begin{itemize}
\item For  $ |y_1|<r$ or $|y_1|\geq 2r,$ if $\inf_{G(y,r)} u\leq 1$ then
\[
\begin{split}
\left|\left\{
u\leq M
\right\}\cap G(y,r)\right|> \nu|G(y,r) |,
\end{split}
\]
i.e.
\[
\begin{split}
\left|\left\{
u> M
\right\}\cap G(y,r)\right|< (1- \nu)|G(y,r) |.
\end{split}
\]

\item
For $r\leq |y_1|<2r,$
if $\inf_{G(y,r)} u\leq 1$ or $\inf_{G(y,r)} u_S\leq 1$ then
\[
\begin{split}
\left|\left\{
u\leq M
\right\}\cap G(y,r)\right|+\left|\left\{
u_S\leq M
\right\}\cap G(y,r)\right|>2 \nu|G(y,r) |,
\end{split}
\]
i.e.
\[
\begin{split}
\left|\left\{
u> M
\right\}\cap G(y,r)\right|+\left|\left\{
u_S> M
\right\}\cap G(y,r)\right|< 2 (1-\nu)|G(y,r) |.
\end{split}
\]

\end{itemize}

By taking the negation of the previous implications we  have, respectively 
\begin{itemize}

\item
For $|y_1|<r$ or $|y_1|\geq 2r,$ if
\[
\begin{split}
\left|\left\{
u> M
\right\}\cap G(y,r)\right|\geq  (1-\nu) |G(y,r) |
\end{split}
\]
then $\inf_{G(y,r)} u >1.$
\item
 For  $r\leq |y_1|<2r,$ if
\begin{equation}\label{uS}
\begin{split}
\left|\left\{
u> M
\right\}\cap G(y,r)\right|+\left|\left\{
u_S> M
\right\}\cap G(y,r)\right|\geq 2 (1-\nu)|G(y,r) |,
\end{split}
\end{equation}
then $\inf_{G(y,r)} u>1$ and  $\inf_{G(y,r)} u_S> 1.$
\end{itemize}
\end{remark}
Let us define $\tilde B(y,r)=B(y,r) \cup B(S(y),r),$ with $S$ as in \eqref{S} and $B$ as in \eqref{ball}.
By recalling the structure Theorem \ref{equiv} and Theorem \ref{charact}  and by rescaling we have
\begin{corollary}\label{cor}
There exist universal constants  $0<\epsilon,\theta<1$ and $M>1,\eta>2$ such that, for all $y\in \R^2$ and $r>0,$
if $u\in C^2(B(y, \eta r))\cap C(\overline{B(y,\eta r)})$ is a nonnegative solution of $Lu\leq 0$ in $B(y,\eta r)$ and
\[
\left|\left\{
u\geq M
\right\}\cap \tilde B(y,r)\cap B(y,\eta r)\right|\geq  \epsilon |\tilde B(y,r)\cap B(y,\eta r) |,
\]
then 
$\inf_{\tilde B(y,\theta r)\cap B(y,\eta r)} u>1.$
\end{corollary}

\section{Double Ball Property, Power Decay Property and Harnack's Inequality}\label{dbp}

We start with  the definition of a  uniform lower barrier function for a ring and for the operator $L.$
\begin{definition}
Let $0<\gamma<1.$ 
A function $\Phi$ is a $\gamma$-lower barrier function for the ring $R(y,r,3r):=\tilde G(y,3r)\setminus \overline{\tilde G(y,r)}$ for every  $y\in \R^2$ and $r>0$ and for the operator $L$  if
\begin{itemize}
\item $\Phi\in C^2(R(y,r,3r))\cap C(\overline{R(y,r,3r)})$ 
\item $L \Phi \geq 0$ on $R(y,r,3r),$
\item $\Phi|_{\p \tilde G(y,3r)}\leq 0$
\item $\Phi|_{\p \tilde G(y,r)}\leq 1.$
\item $\inf_{\p \tilde G(y,2r)}\Phi\geq \gamma.$
\end{itemize}
\end{definition}
The main tool of this section is the following Ring Theorem, which has independent interest because you can reproduce it whenever 
you have  the weak maximum principle.

\begin{theorem}[Ring Theorem]\label{rt} 
Suppose  there exists $\Phi$ a  $\gamma$-barrier function for the ring $R(y,r,3r)=\tilde G(y,3r)\setminus \overline{\tilde G(y,r)}$ for every  $y\in \R^2$ and $r>0$ and for the operator $L.$
Then the Double Ball Property  (see Definition \ref{tralli}) 
holds true in $\tilde G(y,3r)$ with constant $\gamma.$

\end{theorem}
\begin{proof}
The main ingredient of the proof is the weak maximum principle, Theorem \ref{WMP}.
Let $u\in C^2(\tilde G(y,3r))\cap C(\overline{\tilde G(y,3r)})$ be a nonnegative classical solution of $Lu\leq 0$ in $\tilde G(y,3r)$ and assume $u\geq 1$ in $\tilde G(y,r).$
Let $\Phi$ be a $\gamma$-lower barrier function for the ring $R(y,r,3r)=\tilde G(y,3r)\setminus \overline{\tilde G(y,r)}.$
By the weak maximum principle we have $u\geq \Phi$ in the ring $R(y,r,3r).$
In particular, $u \geq \inf_{\p \tilde G(y,2r)}\Phi$ on  $\p \tilde G(y,2r).$
Now consider the function $u$ in $\tilde G(y,2r).$ We have $Lu\leq 0$ in $\tilde G(y,2r)$ and $u\geq \inf_{\p \tilde G(y,2r)}\Phi\geq\gamma$ on $\p \tilde G(y,2r).$
Then by the weak maximum principle we have $u\geq \gamma$ in $\tilde G(y,2r).$
\end{proof}

In the following proposition we prove the existence of a $\gamma$-lower barrier function for a ring and for the operator $L$.

\begin{proposition}\label{barrier}
There exists $\gamma>0$ such that,
for every $y\in \R^2$ and  $r>0,$ there exists $\Phi$ a  $\gamma$- lower barrier function for the ring $\tilde G(y,3r)\setminus  \overline{\tilde G(y,r)}$  and for the operator $L$ in \eqref{L}. Moreover, $\Phi$ is an even function with respect to $x_1.$
\end{proposition}
\begin{proof}
Let $\rho$ be the function in \eqref{rho} and let $\phi=\rho^{\a}$ be the function in Lemma \ref{Lbarrier}.
Recall that $\a\leq 2-3 \L/\l$ and therefore $\a\leq -1.$
Now take $\Phi =M_2 \phi -M_1,$ and choose $M_1,M_2$ such that $\Phi|_{\p \tilde G(y,3r)}= 0$
and $\Phi|_{\p \tilde G(y,r)}= 1.$ 
Obviously, $\Phi$ is an even function with respect to $x_1.$
We will show that $M_1,M_2$ are positive.
We distinguish four cases.

\noindent
I CASE. If $|y_1|<r$ we have 
\[
M_1= \frac{3^\a}{1-3^\a}>0, \quad M_2=\frac{1}{r^\a(1-3^\a)}>0.
\]
Let us put
\begin{equation}\label{m3I}
M_3=\Phi|_{\p \tilde G(y,2r)}=\frac{2^\a-3^\a}{1-3^\a}>0.
\end{equation}

\noindent
II CASE. If $3r\leq |y_1|,$ we have 
\[
M_1= \frac{3^{\a/2}}{1-3^{\a/2}}>0, \quad M_2=\frac{1}{(r|y_1|)^{\a/2}(1-3^{\a/2})}>0.
\]
Let us put
\begin{equation}\label{m3II}
M_3=\Phi|_{\p \tilde G(y,2r)}=\frac{2^{\a/2}-3^{\a/2}}{1-3^{\a/2}}>0.
\end{equation}

\noindent
III CASE. If $r\leq |y_1|<2r,$ we have  
\[
M_1= \frac{3^{\a}}{(|y_1|/r)^{\a/2}-3^{\a}}>0, \quad M_2=\frac{1}{(r|y_1|)^{\a/2}-(3r)^\a}>0.
\]
Let us put
\begin{equation}\label{m3III}
M_3=\Phi|_{\p \tilde G(y,2r)}=\frac{2^{\a}-3^{\a}}{(|y_1|/r)^{\a/2}-3^{\a}}\geq \frac{2^\a-3^\a}{1-3^a}>0.
\end{equation}

\noindent
IV CASE. If $2r\leq |y_1|<3r,$ we have 
\[
M_1= \frac{3^{\a}}{(|y_1|/r)^{\a/2}-3^{\a}}>0, \quad M_2=\frac{1}{(r|y_1|)^{\a/2}-(3r)^{\a}}>0.
\]
Let us put
\begin{equation}\label{m3IV}
M_3=\Phi|_{\p \tilde G(y,2r)}=\frac{(2|y_1|/r)^{\a/2}-3^{\a}}{(|y_1|/r)^{\a/2}-3^{\a}}\geq \frac{6^{\a/2}-3^\a}{2^{\a/2}-3^\a}>0.
\end{equation}
Now choose 
\[
\gamma=\min\left\{  \frac{2^\a-3^\a}{1-3^\a}, \frac{2^{\a/2}-3^{\a/2}}{1-3^{\a/2}}, \frac{6^{\a/2}-3^\a}{2^{\a/2}-3^\a} \right\}.
\]
By recalling Lemma \ref{Lbarrier} and \eqref{m3I}, \eqref{m3II},\eqref{m3III},\eqref{m3IV}
we have that $\Phi$ is a $\gamma$-lower barrier function for the ring $\tilde G(y,3r)\setminus  \overline{\tilde G(y,r)}$ for every $y\in \R^2$ and  $r>0$ and for the operator $L$ in \eqref{L}.

\end{proof}

As a consequence of Theorem \ref{rt} and of Proposition \ref{barrier} we get

\begin{corollary}\label{rtcor} 
The Double Ball Property for $L$ holds true in the following families of sets
\begin{itemize}
\item[i)] $\tilde G(y,3r),$  for every $y\in \R^2$ and $r>0.$

\item[ii)]  $G(y,3r),$ for $|y_1|<r$ or for $|y_1|\geq 3r.$  

\item[iii)] $\tilde B(y,\eta r)=B(y,\eta r)\cup B(S(y),\eta r)$  for every $y\in \R^2$ and $r>0.$
Here $\eta >2$ is a universal constant and $S$ is the reflexion in \eqref{S}.

\item[iv)]  $B(y,\eta r),$ for $|y_1|<r$ or for $|y_1|\geq \eta r$ and with $\eta>2$ as in iii).  

\end{itemize}
\end{corollary}

\begin{proof}
By Proposition \ref{barrier}  there exists $\Phi$ a  $\gamma$-lower barrier function for the ring $$R(y,r,3r)=\tilde G(y,3r)\setminus \overline{\tilde G(y,r)}$$
 for every  $y\in \R^2$ and $r>0$ and for the operator $L.$ By  Theorem \ref{rt} we get i).
 
To prove ii) we distinguish two cases.

If $|y_1|<r$ we have $\tilde G(y,r)=G(y,r), \tilde G(y,2r)=G(y,2r), \tilde G(y,3r)=G(y,3r),$ and we apply Theorem \ref{rt}.

If $3r\geq |y_1|,$ we have $\tilde G(y,r)\cap \{ x_1y_1\geq 0\}=G(y,r),$
 $\tilde G(y,2r)\cap \{ x_1y_1\geq 0\}=G(y,2r),$  $\tilde G(y,3r)\cap \{ x_1y_1\geq 0\}=G(y,3r),$ and we apply Theorem \ref{rt} in the halfplane $\{ x_1y_1\geq 0\}.$

Statement iii)  follows  from i), by Theorem \ref{equiv} and Theorem \ref{charact} and by rescaling.

Statement  iv) follows  from iii) by taking into account the geometry of the sets $\tilde B.$
\end{proof}

Remark that
 if $r\leq |y_1|<2r,$ we have $$\tilde G(y,r)\cap \{ x_1y_1\geq 0\}=G(y,r), \tilde G(y,2r)=G(y,2r), \tilde G(y,3r)=G(y,3r).$$
  If $2r\leq |y_1|<3r,$ we have $$\tilde G(y,r)\cap \{ x_1y_1\geq 0\}=G(y,r), \tilde G(y,2r)\cap \{ x_1y_1\geq 0\}=G(y,2r), \tilde G(y,3r)=G(y,3r).$$
 In both cases the geometry of the level sets $G$ changes in passing trough the $x_2$ axis.
Our strategy  to overcome this technical problem is to combine Corollary \ref{rtcor} with Corollary \ref{cor} to directly prove
 the following power decay property.

\begin{theorem}[Power Decay Property]\label{pd}
There exist universal constants $\eta >2,$  $0<\epsilon<1$ and $M>1$ such that, for all $y\in \R^2$ and $r>0,$
if $u\in C^2(B(y, \eta r))\cap C(\overline{B(y,\eta  r)})$ is a nonnegative solution of $Lu\leq 0$ in $B(y,\eta  r)$ and
$\inf_{\tilde B(y,r)\cap B(y,\eta r)} u\leq 1,$ then for every $k\in \N$ we have
\[
\left|\left\{
u\geq M^k
\right\}\cap \tilde B(y,r/2)\cap B(y,\eta r)\right|\leq \epsilon^k |\tilde B(y,r/2)\cap B(y,\eta r) |.
\]

\end{theorem}
The following Lemma is a crucial tool in the proof of Theorem \ref{pd}.
\begin{lemma}\label{crucial}
There exist $\eta>2,$ $0<\epsilon<1,$ $M_1>1$ such that for all $M>0,$ for all $y\in \R^2$ and $r>0,$
if $u\in C^2(B(y, \eta r))\cap C(\overline{B(y,\eta  r)})$ is a nonnegative solution of $Lu\leq 0$ in $B(y,\eta  r)$ 
with
\[
\left|\left\{
u\geq M
\right\}\cap \tilde B(y,r)\cap B(y,\eta r)\right|\geq  \epsilon  |\tilde B(y,r)\cap B(y,\eta r) |
\]
then $\inf_{\tilde B(y,r)\cap B(y,\eta r)} u>M/M_1.$
\end{lemma}
\begin{proof}
Let $M_0>1,$ $\eta>2,$ $0<\epsilon, \theta<1$ be the constants in Corollary \ref{cor}. Since
\[
\left|\left\{
M_0 u/M \geq M_0
\right\}\cap \tilde B(y,r)\cap B(y,\eta r)\right|\geq  \epsilon  |\tilde B(y,r)\cap B(y,\eta r) |
\]
then $\inf_{\tilde B(y,\theta r)\cap B(y,\eta r)}>M/M_0.$

By eventually enlarging $\eta,$ assume that Corollary \ref{rtcor} iii) holds true and that we can choose $k\in \N$ such that
$1\leq 2^k\theta<\eta.$  Let us call $\gamma$ the constant in the Double Ball Property. By iterating $k$ times Corollary \ref{rtcor} iii) we get 
$u>M \gamma^k/M_0$ in $\tilde B(y, 2^k \theta r)\cap B(y,\eta r).$ In particular, $u>M \gamma^k/M_0$ in $\tilde B(y,  r)\cap B(y,\eta r).$
Now take $M_1=M_0/\gamma^k.$
\end{proof}

\begin{proof}[Proof of Theorem \ref{pd}]
Let $\eta>2,M>1$ be universal constants such that Lemma \ref{crucial} and Corollary \ref{rtcor} iii) hold true.
Define $E_k=\left\{
u\geq M^k
\right\}\cap B(y,\eta r).$
By following the proof of  \cite[Theorem 4.7, conditions A1 and A2]{DGL}   and by recalling Theorem \ref{ringc} and Lemma \ref{crucial},
construct a family of quasi metric balls 
$B_k=B(y,t_k)$ with $t_0=r>t_1>t_2>\dots>r/2$ and choose $0<\epsilon<1$ such that
\[
\left|
\tilde B_{k+1}\cap E_{k+2}
\right|\leq  \epsilon \left|
\tilde B_{k}\cap E_{k+1}
\right|,\quad \forall k=0,1,\dots.
\]
We have
\[
\begin{split}
\left|\left\{
u\geq M^{k+2}
\right\}\cap \tilde B(y,r/2)\cap B(y,\eta r)\right|&\leq \epsilon^{k+1} |\tilde B(y,r)\cap B(y,\eta r) |\\
&\leq C_D\epsilon^{k+1} |\tilde B(y,r/2)\cap B(y,\eta r) |,
\end{split}
\]
where $C_D$ is the doubling constant in \eqref{dconstant}.
Now choose a positive integer $k_0$ such that $\epsilon^{k_0}{C_D}<1,$ and replace $M$ with $M^{2+k_0}$ to get the thesis.

\end{proof}

By Theorem \ref{pd} we immediately get the following corollaries.
 
\begin{corollary}\label{2.5tralli} There is a positive constant $\eta>2$ such that
the Double Ball Property for $L$ holds true in 
$B(y,\eta r)$ for all $y\in \R^2$ and $r>0.$ 
 \end{corollary}
 \begin{proof}
 Let $\eta >2$ be such that  Theorem \ref{pd} and Corollary \ref{rtcor} iii) hold true.
 Let $u\in C^2(B(y,\eta r))\cap C(\overline{B(y,\eta r)})$ 
 be a nonnegative classical solution of $Lu\leq 0$ in $B(y,\eta r)$ and assume $u\geq 1$ in $B(y,r).$ 
 If $|y_1|<r$ or $|y_1|\geq \eta r$ we apply Corollary \ref{rtcor} vi).
 
If $r\leq |y_1|< \eta r,$
let $M>1$ and $0<\epsilon<1$ be the universal constants in Theorem \ref{pd}. Now choose 
 $k\in \N$ such that $\epsilon^k <1/4.$
 If $u\geq 1/M^k$ in $\tilde B(y,r)$ then by Corollary \ref{rtcor} iii) $u\geq \gamma/M^k$ in $\tilde B(y,2r).$ 
 In particular, 
 $u\geq \gamma/M^k$ in $B(y,  2r)$ and the thesis follows.
 
  On the contrary, if there is a point $x_0\in \tilde B(y,r)$ such that $M^ku(x_0)<1$ then by considering the symmetry $S$ in \eqref{S}  and by Theorem 
 \ref{pd} and \eqref{uS} we get
  \[
 |\{u\geq 1\}\cap B(y,r/2)|+ |\{u_S\geq 1\}\cap B(y,r/2)|\leq 2 \epsilon^k |B(y,r/2)|< (1/2) |B(y,r/2)|
 \]
 and by recalling that $u\geq 1$ in $B(y,r),$ in particular $u\geq 1$ in $B(y,r/2),$ and 
 \[
 |B(y,r/2)|\leq |\{u\geq 1\}\cap B(y,r/2)|+ |\{u_S\geq 1\}\cap B(y,r/2)|<(1/2) |B(y,r/2)|
 \]
 we get a contradiction. 
 
 \end{proof}
 
\begin{corollary}\label{pdlemma}
There exist universal constants $\eta >2,$  $0<\epsilon<1$ and $M>1$ such that, for all $y\in \R^2$ and $r>0,$
if $u\in C^2(B(y, \eta r))\cap C(\overline{B(y,\eta  r)})$ is a nonnegative solution of $Lu\leq 0$ in $B(y,\eta  r)$ and
$\inf_{ B(y,r)\cap B(y,\eta r)} u\leq 1,$ then for every $k\in \N$ we have
\[
\left|\left\{
u\geq M^k
\right\}\cap  B(y,r/2)\cap B(y,\eta r)\right|\leq \epsilon^k |B(y,r/2)\cap B(y,\eta r) |.
\]

\end{corollary}
\begin{proof}
In particular, $\inf_{\tilde B(y,r)\cap B(y,\eta r)} u\leq 1,$ and by Theorem \ref{pd} we get
\[
\left|\left\{
u\geq M^k
\right\}\cap  B(y,r/2)\cap B(y,\eta r)\right|\leq 2 \epsilon^k |B(y,r/2)\cap B(y,\eta r) |.
\]
Now choose a positive integer $k_0$ such that $\epsilon^{k_0}<1/2,$ and replace $M$ with $M^{1+k_0}$ to get the thesis.

\end{proof}
An alternative  proof of Corollary \ref{pdlemma} can be obtained
by applying Corollary \ref{2.5tralli}, the critical density estimate in Theorem \ref{cdB} and the results of Di Fazio et al. \cite[Theorem 4.7] {DGL}.

By Corollary \ref{pdlemma} and the results of Di Fazio et al. \cite[Theorem 5.1] {DGL}
applied to 
\[
\begin{split}
K=& \{
u\in C^2(B(y,\eta r))\cap C(\overline{B(y,\eta r)}): u\geq 0 \, \textrm{and} \, Lu\leq 0 \, \textrm{in}\, B(y,\eta r),\, u\geq 1\,  \textrm{on} \, B(y,r)
\}
\end{split}
\]
we obtain the following invariant Harnack inequality.
\begin{theorem}[Harnack inequality]\label{Har}
There exist constants $C$ and $\eta,$ both bigger than $1$ and
 depending only on the ellipticity constants, such that  for every $y\in \R^2$ and $r>0,$ if $Lu = 0$ and $u \geq 0$ in $B(y, \eta r),$ then
 \[
\sup_{B(y, r)} u\leq C \inf_{B(y, r)} u.
\]
\end{theorem}
The scale invariant Harnack's inequality on balls $B_{CC}$ easy follows from Theorem \ref{Har}  and from Theorem \ref{FRLa} and Theorem \ref{charact}.

\section*{Acknowledgements} 

It is a pleasure to thank Ermanno
Lanconelli, Daniele Morbidelli and Cristian Gutierrez for several useful discussions during the preparation of this paper.

The author is member of the {\it Gruppo Nazionale per
l'Analisi Matematica, la Probabilit\`a e le loro Applicazioni} (GNAMPA)
of the {\it Istituto Nazionale di Alta Matematica} (INdAM)

\end{document}